\documentclass[12pt]{amsart}

\usepackage{amsmath} 
\usepackage{amssymb} 
\usepackage{amsthm}
\usepackage{amscd}
\usepackage{enumerate}

\def\a{{\mathfrak{a}}} \def\b{{\mathfrak{b}}}    
\def\F{{\mathbb{F}}} \def\J{{\mathcal{J}}} \def\m{{\mathfrak{m}}}\def\n{{\mathfrak{n}}} \def\Z{{\mathbb{Z}}}\def\N{{\mathbb{N}}}
\def\sO{{\mathcal{O}}}
\def\Q{{\mathbb{Q}}}  
\def\Hom{{\mathrm{Hom}}}  
  \def\Ann{{\mathrm{Ann}}}
\def\Spec{{\mathrm{Spec\; }}}

\def\Tor{\operatorname{Tor}}
\def\MM{{\mathfrak{M}}}

\DeclareMathOperator{\pt}{c}

\theoremstyle{plain}
\newtheorem{thm}{Theorem}[section]
\newtheorem{disc}[thm]{Discussion}
\newtheorem{cor}[thm]{Corollary}
\newtheorem{prop}[thm]{Proposition}
\newtheorem{conj}[thm]{Conjecture}

\newtheorem{lem}[thm]{Lemma}
\theoremstyle{definition}
\newtheorem{defn}[thm]{Definition}
\newtheorem{eg}[thm]{Example}
\theoremstyle{remark}
\newtheorem{rem}[thm]{Remark}
\newtheorem{ques}[thm]{Question}

\newtheorem{notation}[thm]{Notation}
\newtheorem*{cl}{Claim}

\newtheorem*{acknowledgement}{Acknowledgments}

\begin{document}

\title{Multiplicity bounds in graded rings}
\author[Huneke]{Craig Huneke}
\address{Department of Mathematics, University of Kansas,
Lawrence, KS 66045-7523, USA} \email{huneke@math.ku.edu}

\author[Takagi]{Shunsuke Takagi}
\address{Department of Mathematics, Kyushu University, 744 Motooka, Nishi-ku, Fukuoka 819-0395, Japan}
\email{stakagi@math.kyushu-u.ac.jp}

\author[Watanabe]{Kei-ichi Watanabe}
\address{Department of Mathematics, College of Humanities and Sciences,
Nihon University, Setagaya-Ku, Tokyo 156-0045, Japan}
\email{watanabe@math.chs.nihon-u.ac.jp}


\baselineskip 15pt \footskip = 32pt

\begin{abstract}
The $F$-threshold $\pt^J(\a)$ of an ideal $\a$ with respect to an ideal $J$ is a positive characteristic invariant obtained by comparing the powers of $\a$ with the Frobenius powers of $J$. 
We study a conjecture formulated in an earlier paper \cite{HMTW} by the same authors together with M. Musta\c{t}\u{a},  which bounds $\pt^J(\a)$ in terms of the multiplicities $e(\a)$ and $e(J)$, when $\a$ and $J$ are zero-dimensional ideals and $J$ is generated by a system of parameters. 
We prove the conjecture when $\a$ and $J$ are generated by homogeneous systems of parameters in a Noetherian graded $k$-algebra. 
We also prove a similar inequality involving, instead of the $F$-threshold,  the jumping number for the generalized parameter test submodules introduced in \cite{ST}. 
%We also prove that the $F$-thresholds exist in greater generality than was previously known.
\end{abstract}

\maketitle \markboth{C.~Huneke,  S.~Takagi, and K.-i.~Watanabe}{Multiplicity bounds in graded rings}

\section*{Introduction}

Let $R$ be a Noetherian ring of prime characteristic $p$. 
For every ideal $\a$ in $R$ and for every ideal $J$ whose radical contains $\a$, one can define asymptotic invariants that measure the containment of the powers of $\a$ in the Frobenius powers of $J$.
These invariants, dubbed $F$-thresholds, were introduced in the case of a regular local ring in \cite{MTW}, %where it was shown that they coincide with the jumping numbers for the generalized test ideals of Hara and Yoshida \cite{HY}, 
and in full generality in the paper \cite{HMTW}.
In this paper we work in the general setting.

A conjecture was made in \cite{HMTW} which connects $F$-thresholds with the multiplicities of the ideals $\a$ and $J$ (see Conjecture~\ref{firstconj} below).  
A second question was stated in the same paper which does not explicitly refer to $F$-thresholds but which implies Conjecture~\ref{firstconj}. This second question is easy to state:

\begin{ques}\label{introques}
Suppose that $\a$ and $J$ are $\m$-primary ideals in a $d$-dimensional Noetherian local or Noetherian graded $k$-algebra $(R, \m)$, where $k$ is a field of arbitrary characteristic. 
Further suppose that $J$ is generated by  a system of parameters (homogeneous in the graded case).  
If $\a^{N+1}\subseteq J$ for some integer $N \ge 0$, then
$$e(\a)\geq\left(\frac{d}{d+N}\right)^d e(J)?$$
Here $e(K)$ denotes the multiplicity of the ideal $K$.
\end{ques}

In \cite[Theorem 5.8]{HMTW} this question was answered in the affirmative if $R$ is a Cohen-Macaulay graded ring and $\a$ is generated by a homogenous system of parameters.  
In this paper we generalize this result to the case of arbitrary graded rings:

\begin{thm}[Theorem \ref{mainthm}]\label{introthm}
The answer to Question \ref{introques} is yes if $R$ is a graded ring and $\a$ is generated by a homogeneous system of parameters for $R$. 
\end{thm}

In fact we prove that if $R$ is in addition a domain of positive characteristic, then the power $N$ in the statement of Question \ref{introques} can be changed to be the least integer $N \ge 0$ such that $\a^{N+1}\subseteq J^{+\mathrm{gr}}$, where $J^{+\mathrm{gr}}$ is a certain graded plus closure of the ideal $J$ (see Discussion~\ref{rplusdef} below). 
This result not only removes the Cohen-Macaulay assumption on $R$, but even strengthens the previous result in case $R$ is a Cohen-Macaulay graded domain of positive characteristic.  
The proof of Theorem \ref{introthm} uses reduction to characteristic $p>0$, and takes advantage of the fact that the graded plus closure of a graded Noetherian domain in positive characteristic is a big Cohen-Macaulay algebra (see \cite{HH}).

%We also show the equivalence of the original conjecture in \cite{HMTW} with a weaker conjecture, correcting an error in that paper. 
%We also prove that $F$-thresholds actually exist when the ring is $F$-pure away from $\a$.

Another ingredient of this paper is a comparison of $F$-thresholds and $F$-jumping numbers, jumping numbers for the generalized parameter test submodules introduced in \cite{ST} (see Definition \ref{parameter $F$-jumping} for their definition). 
In \cite[Proposition 2.7]{MTW} and \cite[Remark 2.5]{HMTW}, it was shown that those two invariants coincide with each other in case the ring is regular, but not in general. 
In this paper, we show that if the ideal $J$ is generated by a full system of parameters and if the ring is $F$-rational away from $\a$, then the $F$-jumping number $\mathrm{fjn}^J(\omega_R, \a)$ coincides with $F$-threshold $\pt^J(\a)$. 
Also, when $\a$ and $J$ are ideals generated by full homogeneous systems of parameters in a Noetherian graded domain $R$ over a field of positive characteristic, 
we prove an inequality similar to that of the original conjecture in \cite{HMTW} (Conjecture \ref{firstconj}), involving the $F$-jumping number $\mathrm{fjn}^J(\omega_R, \a)$  instead of the $F$-threshold $\pt^J(\a)$. 

The paper is structured as follows. In \S1 we recall some notions needed throughout the rest of the paper, and introduce basic facts about $F$-thresholds. 
We prove their existence in a new case in Theorem~\ref{fpurity}. 
In \S 2 we prove our main result, Theorem~\ref{mainthm}, in positive characteristic,
%Along the way we need to prove an apparently new result concerning the multiplicity of ideals in graded rings. See Proposition~\ref{mult}.
and in \S 3 we prove the case of characteristic zero. 
%the main theorem for graded rings over fields of characteristic zero.
In \S4 we compare $F$-thresholds and $F$-jumping numbers. 
We give a lower bound on the $F$-jumping number in terms of multiplicities in Corollary \ref{maincor2}. 

\bigskip

\section{Preliminaries and Basic Results}

\medskip

In this section we review some definitions and notation that will be used throughout the paper, and state and prove basic results on $F$-thresholds. 
All rings are Noetherian commutative rings with unity unless explicitly stated otherwise.  
A particularly important exception is the graded plus closure of a graded Noetherian domain, which essentially is never Noetherian.
For a ring $R$, we denote by $R^{\circ}$ the set of elements of $R$ that are not contained in any minimal prime ideal. 
Elements $x_1,\dots, x_r$ in $R$ are called \emph{parameters} if they generate an ideal of height $r$. 

For a real number $u$, we denote by $\lfloor u\rfloor$ the largest integer $\leq u$, and by $\lceil u \rceil$ the smallest integer $\geq u$.

\begin{defn}
Let $R$ be a ring of prime characteristic $p$.
\renewcommand{\labelenumi}{(\roman{enumi})}
\begin{enumerate}
\item 
We always let $q = p^e$ denote a power of $p$. If $I$ is an ideal of $R$, then $I^{[q]}$ is the ideal generated by the set of all $i^q$ for $i\in I$. The {\it Frobenius closure} $I^F$ of $I$ is defined as the ideal of $R$ consisting of all elements $x \in R$ such that $x^q \in I^{[q]}$ for some $q=p^e$.

\item $R$ is {\it $F$-finite} if the Frobenius map $R\rightarrow R$ sending $r$ to $r^p$ is a finite map.

\item $R$ is {\it $F$-pure} if the Frobenius map $R\rightarrow R$ is pure. 

\end{enumerate}
\end{defn}

Let $R$ be a Noetherian ring of dimension $d$ and of characteristic $p>0$. 
Let $\a$ be a fixed proper ideal of $R$ such that $\a \cap R^{\circ} \ne \emptyset$. 
To each ideal $J$ of $R$ such that $\a\subseteq \sqrt{J}$, we associate the $F$-threshold $\pt^J(\a)$ as follows. 
For every $q=p^e$, let
$$\nu_{\a}^J(q):=\max\{r\in\N\vert\a^r\not\subseteq J^{[q]}\}.$$
These numbers  can be thought of as characteristic $p$ analogues of Samuel's asymptotic function (for example, see \cite[6.9.1]{SH}).
Since $\a\subseteq \sqrt{J}$, these  are nonnegative integers (if $\a \subseteq J^{[q]}$, then we put $\nu_{\a}^J(q)=0$). 
We put
$$\pt_{+}^J(\a)=\limsup_{q\to\infty}\frac{\nu_{\a}^J(q)}{q}, \quad \pt_{-}^J(\a)=\liminf_{q \to \infty}\frac{\nu_{\a}^J(q)}{q}.$$
When $\pt_{+}^J(\a)=\pt_{-}^J(\a)$, we call this limit the \emph{$F$-threshold} of the pair $(R,\a)$ (or simply of $\a$) with respect to $J$, and we denote it by $\pt^J(\a)$.
We refer to \cite{BMS2}, \cite{BMS}, \cite{HM},   
\cite{HMTW} and \cite{MTW} 
for further information.

\begin{rem}[{cf.\! \cite[Remark 1.2]{MTW}}]\label{rem1}
One has
$$0 \le \pt_-^J(\a) \le \pt_+^J(\a) < \infty.$$
\end{rem}

\begin{rem}[\textup{\cite[Proposition 2.2]{HMTW}}]\label{rem2}
Let $\a$, $J$ be ideals as above.
\begin{enumerate}
\item If $I\supseteq J$, then $\pt_{\pm}^I(\a) \leq \pt_{\pm}^J(\a)$.
\item If $\b \subseteq \a$, then $\pt_{\pm}^J(\b) \leq\pt_{\pm}^J(\a)$.
Moreover, if $\a\subseteq \overline{\b}$, then $\pt_{\pm}^J(\b)
= \pt_{\pm}^J(\a)$. Here, $\overline{\b}$ denotes the integral closure of $\b$.
\item $\pt_{\pm}^I(\a^r) = \frac{1}{r} \pt_{\pm}^I(\a)$ for every integer $r \ge 1$.
\item  $\pt_{\pm}^{J^{[q]}}(\a) =q\pt_{\pm}^J(\a)$ for every $q=p^e$.
\item $\pt_+^J(\a) \le c$ $($resp. $\pt_-^J(\a) \ge c)$ if and only if for every power $q_0$ of $p$, we have
$\a^{\lceil cq \rceil+q/q_0} \subseteq J^{[q]}$ $($resp. $\a^{\lceil
cq \rceil-q/q_0} \not\subseteq J^{[q]})$ for all $q=p^e \gg q_0$.
\end{enumerate}
\end{rem}

The $F$-threshold $\pt^J(\a)$ exists in many cases.

\begin{lem}[\textup{\cite[Lemma 2.3]{HMTW}}]\label{lem1} 
Let $\a$, $J$ be as above.
\renewcommand{\labelenumi}{$(\arabic{enumi})$}
\begin{enumerate}
\item If $J^{[q]}=(J^{[q]})^F$ for all large $q=p^e$, then the $F$-threshold $\pt^J(\a)$ exists, that is, $\pt_{+}^J(\a)=\pt_{-}^J(\a)$. In particular, if $R$ is $F$-pure, then $\pt^J(\a)$ exists.

\item If $\a$ is principal, then $c^J(\a)$ exists.
\end{enumerate}
\end{lem}

\begin{rem}\label{remark-reduced} Let $\a$, $J$ be as above. Set $R_{\mathrm{red}}:=R/\sqrt{(0)}$
and let $\overline{\a}, \overline{J}$ be the images of $\a$ and $J$  in $R_{\mathrm{red}}$, respectively. Let
$\mu$ be the least number of generators for the ideal $\a$ and $q'$ be a power of $p$ such
that $\sqrt{(0)}^{[q']}=0$. Then for all $q=p^e$,
$\nu^{\overline{J}}_{\overline{\a}}(qq')/qq' \leq \nu^J_{\a}(qq')/qq'
\leq \nu^{\overline{J}}_{\overline{\a}}(q)/q + \mu/q$.  In particular, if the $F$-threshold $\pt^{\overline{J}}(\overline{\a})$ exists, then the $F$-threshold $\pt^J(\a)$ also exists.
\end{rem} 

Our first new result is that the $F$-threshold exists more generally when the ring is $F$-pure away from the ideal $\a$. 

\begin{thm}\label{fpurity} Let $(R,\m)$ be a local F-finite Noetherian ring of characteristic $p$, and
let $\a$ and $J$ be ideals such that the radical  of $J$ contains $\a$. Assume
that $R_P$ is $F$-pure for all primes $P$ which do not contain $\a$. Then the $F$-threshold $\pt^J(\a)$ exists, that is, $\pt_{+}^J(\a)=\pt_{-}^J(\a)$.
\end{thm}

\begin{proof}

We use the following notation: $^eR$ is $R$ thought of
as an $R$-module via the $e^{\mathrm{th}}$ iterate of the Frobenius map $F$. 
Thus the $R$-module structure on $^eR$ is given by $r\cdot s = r^qs$ for $r\in R$, $s\in$  $^eR$, and $q = p^e$.

By Remark~\ref{remark-reduced}, we may assume that $R$ is reduced.
The map from $R_P$ to $^1(R_P)$ is split for every prime $P$ not containing $\a$. 
Since the Frobenius map commutes with localization and $R$ is $F$-finite,
there exists an $R$-homomorphism $f_P:$ $^1R\rightarrow R$ such that $f_P(1) = u_P\notin P$. 
By \cite[Lemma 6.21]{HH}, for each $e\geq 1$ there exists an $R$-linear map from $^eR\rightarrow R$ taking $1$ to $u_P^2$ (in \cite[Lemma 6.21]{HH} it is assumed that the element $u_P\in R^o$, but that is not used in the proof of this lemma). 

Let $I$ be the ideal generated by the set of $u_P^2$ for $P$ ranging over all prime ideals not containing $\a$.
We claim that $I(J^{[q]})^F\subseteq J^{[q]}$ for all $q = p^e$. Suppose that $r\in (J^{[q]})^F$. Then
there exists a power $q' = p^{e'}$ of $p$ such that $r^{q'}\in J^{[qq']}$, and so 
$r\in J^{[q]}($$^{e'}R)$. For each prime $P$ not containing
$\a$, there
is an $R$-linear map $g_P:$ $^{e'}R\rightarrow R$ such that $g_P(1) = u_P^2\notin P$. Then
$ru_P^2 = g_P(r\cdot 1) \in g_P(J^{[q]}($$^{e'}R))\subseteq J^{[q]}$, showing that $Ir\subseteq J^{[q]}$ as claimed. 

Since $u_P\notin P$, $I$ is not contained in any prime $P$ which does not contain $\a$. Hence $\a\subseteq \sqrt{I}$, and
there exists an integer $k$ such that $\a^{k}\subseteq I$. We claim that 
for all powers of $p$, $q$ and $Q$,
$$ \frac {v_{\a}^J(qQ)+1} {qQ} \geq \frac {v_{\a}^J(q)} {q} - \frac{k}{q}.$$
Since the set of values $\{\frac {v_{\a}^J(q)} {q}\}$ is bounded above, this implies
that the limit exists. To prove this claim, 
fix $Q$, and write $\lceil\frac{v_{\a}^J(qQ) + 1}{Q}\rceil  = a$. Then 
$$(\a^a)^{[Q]}\subseteq \a^{aQ}\subseteq \a^{v_{\a}^J(qQ) + 1}\subseteq J^{[qQ]},$$ 
where the last containment follows from the definition of $v_{\a}^J(qQ)$. 
Hence $\a^a\subseteq (J^{[q]})^F$, and the work above shows
that $\a^{a+k}\subseteq I(J^{[q]})^F\subseteq  J^{[q]}$.
It follows that $v_{\a}^J(q)\le a + k - 1$. Since $a \leq \frac {v_{\a}^J(qQ)+1} {Q} + 1$, dividing
by $q$ gives the required inequality and finishes the proof.

\end{proof}

\bigskip

\section{Connections between $F$-thresholds and multiplicity}

\medskip

The following conjecture was proposed in \cite[Conjecture 5.1]{HMTW}:

\begin{conj}\label{firstconj}
Let $(R,\m)$ be a $d$-dimensional Noetherian local ring of
characteristic $p>0$. If $J \subseteq \m$ is an ideal generated by a
full system of parameters, and if $\a \subseteq\m$ is an
$\m$-primary ideal, then
$$e(\a) \ge \left(\frac{d}{\pt_{-}^J(\a)}\right)^d e(J).$$
\end{conj}

Given an $\m$-primary ideal $\a$ in a regular local ring $(R,\m)$,
essentially of finite type over a field of characteristic zero,  de
Fernex, Ein and Musta\c{t}\u{a} proved in \cite{dFEM} an
inequality involving the log canonical threshold $\mathrm{lct}(\a)$
and the multiplicity $e(\a)$. Later, the second and third authors
gave in \cite{TW} a characteristic $p$ analogue of this result,
replacing the log canonical threshold $\mathrm{lct}(\a)$ by the
$F$-pure threshold $\mathrm{fpt}(\a)$. The conjecture above generalizes
these inequalities.

We list some known facts.

\begin{rem}\label{rmkmult}
(1) (\cite[Remark 5.2 (c)]{HMTW})  
The condition in Conjecture~\ref{firstconj} that $J$ is
generated by a system of parameters is necessary.

(2) (\cite[Proposition 5.5]{HMTW}) 
If $(R,\m)$ is a one-dimensional analytically irreducible local domain of characteristic $p>0$,
and if $\a, J$ are $\m$-primary ideals in $R$, then $$\pt^J(\a)=\frac{e(J)}{e(\a)}.$$ In particular, Conjecture $\ref{firstconj}$ holds in $R$.

(3) (\cite[Theorem 5.6]{HMTW}) 
If $(R,\m)$ is a regular local ring of characteristic $p>0$ and
$J=(x_1^{a_1}, \dots, x_d^{a_d})$, with $x_1, \dots, x_d$ a full
regular system of parameters for $R$, and with $a_1, \dots, a_d$
positive integers, then 
Conjecture~$\ref{firstconj}$ holds.

(4) (\cite[Corollary 5.9]{HMTW})
Let $R=\bigoplus_{n \ge 0} R_n$ be a $d$-dimensional graded
Cohen-Macaulay ring with $R_0$ a field of characteristic $p>0$. If $\a$
and $J$ are ideals generated by full homogeneous systems of
parameters for $R$, then
$$e(\a)\ge\left(\frac{d}{\pt_{-}^J(\a)}\right)^d e(J).$$

\end{rem}

The next proposition shows that Conjecture~\ref{firstconj}  actually is equivalent to the special case of itself in which $\pt_{-}^J(\a)\leq d$.

\begin{prop}\label{integral}
Let $(R, \m)$ be a $d$-dimensional formally equidimensional
Noetherian local ring of characteristic $p>0$. 
Then the following are equivalent:
\renewcommand{\labelenumi}{$(\arabic{enumi})$}
\begin{enumerate}
\item $e(\a)\geq e(J)$ for all ideals $J \subseteq R$ generated by full systems of parameters and for all $\m$-primary ideals $\a \subseteq R$ with $\pt^J_-(\a) \le d$. 

\item $e(\a) \ge \left(\frac{d}{\pt_{-}^J(\a)}\right)^d e(J)$ for all ideals $J \subseteq R$ generated by full systems of parameters and for all $\m$-primary ideals $\a \subseteq R$. 
\end{enumerate}
\end{prop}

\begin{proof} First assume (1). Given any positive $\epsilon > 0$, choose $n$
and $q = p^e$ so that $d-\epsilon\leq \frac{q}{n}\pt_{-}^J(\a)\leq d$.
Since $\frac{q}{n}\pt_{-}^J(\a) = \pt_{-}^{J^{[q]}}(\a^n)$ by Remark~\ref{rem2} (3),(4), we get that
$$\pt_{-}^{J^{[q]}}(\a^n)\leq d$$
and then by (1) we obtain
that $n^de(\a) = e(\a^n)\geq e(J^{[q]}) = 
q^de(J)$. Thus we have shown that for any $\epsilon > 0$,
$e(\a)/e(J) \ge (\frac{d-\epsilon}{\pt_{-}^J(\a)})^d$, proving the inequality in (2).

On the other hand, assuming (2) immediately gives (1).\footnote{In \cite[Thm. 3.3.1]{HMTW}, it is incorrectly claimed that $\pt_+^J(I)\leq d$ implies that $I$ is integral over $J$ (which in particular forces $e(I)\ge e(J)$). 
This is false. 
The mistake occurs in the following section of the proof:
``...$\pt_+^J(I)\leq d$ implies that for all $q_0=p^{e_0}$ and for all large $q=p^e$, we have
$I^{q(d+(1/q_0))} \subseteq J^{[q]}.$
Hence $I^q J^{q(d-1+(1/q_0))} \subseteq J^{[q]}$, ..."
This latter statement does not follow unless $J\subseteq I$.}

\end{proof}

\begin{rem}
One can think of the condition (1) in Proposition~\ref{integral} as a converse
to the main point of the tight closure Brian\c con-Skoda Theorem (\cite[Theorem 5.4]{HH}).
Namely, suppose we are in the special case in which $J\subseteq \a$. If $e(\a)\geq e(J)$
then $\a$ is integral over $J$ by a theorem of Rees. In this case there is a constant
$l$ such that for all $q = p^e$, $\a^{qd+l}\subseteq J^{[q]}$. Hence $\pt_{-}^J(\a)\leq d$.
What (1) is claiming in this case is that the converse holds. If $\pt_{-}^J(\a)\leq d$,
then $e(\a)\geq e(J)$, and therefore $\a$ is integral over $J$ (assuming as above that
$J\subseteq \a$). This is interesting as it shows that integrality over parameter
ideals can be detected at the level of Frobenius powers. 
\end{rem}

\begin{rem} In \cite{HMTW}, another problem was raised which
does not depend upon the characteristic, and which implies the conjecture above. We
state it here as a conjecture.  Namely:

\begin{conj}\label{mainconj} 
Suppose that $\a$ and $J$ are $\m$-primary ideals in a $d$-dimensional Noetherian local or Noetherian graded $k$-algebra $(R, \m)$, where $k$ is a field of \emph{arbitrary} characteristic. 
Further assume that $J$ is generated by a full system of parameters $($homogeneous in the graded case$)$. 
If $\a^{N+1}\subseteq J$ for some integer $N \ge 0$, then
$$e(\a)\geq\left(\frac{d}{d+N}\right)^d e(J).$$
\end{conj}

\end{rem}

\begin{rem}\label{reduce} 
Conjecture~\ref{mainconj} reduces to the domain case as follows. 
Let $P_1,...,P_l$ be the minimal primes of $R$ such that the dimension of $R/P_i$ is equal to the dimension of $R$. 
Write $\a_i = (\a+P_i)/P_i$, $J_i = (J+P_i)/P_i$.
Suppose the conjecture holds in each $R/P_i$. The ideal $J_i$ is generated by parameters in $R/P_i$ since the dimension of $R/P_i$ is $d$. 
Moreover, if $\a^{N+1}\subseteq J$, then $\a_i^{N+1}\subseteq J_i$ for each $i$. 
Hence $e(\a_i)\geq\left(\frac{d}{d+N}\right)^d e(J_i).$ 
By \cite[Theorem 11.2.4]{SH}, we then have that 
$$e(\a) = \sum_i e(\a_i)l_R(R_{P_i})\geq \sum_i \left(\frac{d}{d+N}\right)^d e(J_i) l_R(R_{P_i}) = \left(\frac{d}{d+N}\right)^de(J).$$
\end{rem}

\begin{thm}\label{mainthm}
Let $R=\bigoplus_{n \ge 0} R_n$ be a $d$-dimensional Noetherian graded
ring with $R_0$ a field of arbitrary characteristic. 
Suppose that $\a$ $($resp. $J)$ is an ideal generated by a full homogeneous system of
parameters of degrees $a_1 \leq \dots \leq a_d$ $($resp. $b_1 \leq \dots \leq b_d)$ for $R$. 
If $\a^{N+1}\subseteq J$ for some integer $N \ge 0$, then
$$e(\a) \ge \left(\frac{d}{d+N}\right)^d e(J).$$
If the equality holds in the above, then $(a_1,\ldots ,a_d)$ and $(b_1,\ldots ,b_d)$ are proportional, that is, $a_1/b_1=\ldots =a_d/b_d$. 
\end{thm}

As in the proof of \cite[Corollary 5.9]{HMTW} we can immediately obtain that the
first conjecture holds in this case as well:

\begin{cor}\label{maincor}
Let $R=\bigoplus_{n \ge 0} R_n$ be a $d$-dimensional Noetherian graded ring with $R_0$ a field of 
characteristic $p>0$.
Suppose that $\a$ $($resp. $J)$ is an ideal generated by a full homogeneous system of
parameters of degrees $a_1 \leq \dots \leq a_d$ $($resp. $b_1 \leq \dots \leq b_d)$ for $R$. 
Then 
$$e(\a) \ge \left(\frac{d}{\pt_{-}^J(\a)}\right)^d e(J).$$
If the equality holds in the above, then $(a_1,\ldots ,a_d)$ and $(b_1,\ldots ,b_d)$ are proportional, that is, $a_1/b_1=\ldots = a_d/b_d$. 
\end{cor}

\begin{proof} The proof of the former assertion is exactly as in the proof of \cite[Corollary 5.9]{HMTW},
but we repeat it here for the convenience of the reader, since it is quite short.

Note that each $J^{[q]}$ is again generated by a full homogeneous system of parameters. 
It follows from the theorem and from the definition of $\nu^{J}_{\a}(q)$ that for every $q=p^e$ we have
$$e(\a)\geq\left(\frac{d}{d+\nu^J_{\a}(q)}\right)^de(J^{[q]})=\left(\frac{qd}{d+\nu^J_{\a}(q)}\right)^d e(J).$$ 
On the right-hand side we can take a subsequence converging to $\left(\frac{d}{\pt_{-}^J(\a)}\right)^d e(J)$, hence we get the inequality in the corollary. 

For the latter assertion, we postpone the proof to Remark \ref{equality}. 
\end{proof}

The proof of the main theorem is in two parts. First we prove it in characteristic $p$ in this
section, and then do the characteristic zero case in the next section, by reducing to characteristic
$p$.
Before beginning the proof of the main theorem in positive characteristic there are several preliminary results we need to
recall, as well as some new results on multiplicity which we will need.

An important point for us is that the multiplicity of an ideal generated by a full homogeneous system of parameters is determined only by the degrees of the parameters up to a constant which does not depend on the parameters themselves. 
This is well-known in case the ring is Cohen-Macaulay, and a generalization to the non-Cohen-Macaulay case can be found in \cite[Lemma 1.5]{To}. 
We give here a more general statement with a short proof. 

\begin{prop}\label{mult} Let $R=\bigoplus_{n\ge 0} R_n$ be a Noetherian graded ring of dimension
$d$ over
an Artinian local ring $R_0=(A,\m)$ and $\MM=\m R + R_+$ be the unique homogeneous maximal ideal of $R$.  
Let $J$ be an ideal generated by a homogeneous system of parameters $f_1, \dots, f_d$ for $R$. 
Let
$$P(R,t) = \sum_{n\ge 0} l_A(R_n) t^n$$
be the Poincar\'e series of $R$. 
Then the multiplicity $e(J)$ of $J$ is given by
$$e(J)= \deg f_1 \cdots \deg f_d \lim_{t\to 1} (1-t)^d P(R,t)$$
\end{prop}

\begin{proof}
We compute the multiplicity $e(J)$ by using Koszul homology of $J$.\footnote{We thank \textbf{Kazuhiko Kurano} for advising us to use Koszul complex, which makes the proof much simpler.} 
By a theorem of Auslander-Buchsbaum \cite[Theorem 4.1]{AuB} one has $e(J)= \chi(J)$, where $\chi(J)$ denotes the Euler characteristic of the Koszul complex $K_{\bullet}((f_1,\ldots ,f_d), R)$.  

Now let $K_i=K_i((f_1,\ldots ,f_d), R)$ (resp. $H_i=H_i((f_1,\ldots ,f_d), R)$) be the  component 
of degree $i$ (resp. $i^{\mathrm{th}}$ homology module) of the Koszul complex considered as graded $R$-modules, and put $\deg f_i=b_i$ for each $1 \le i \le d$. 
Then the assertion follows from the theorem of Auslander-Buchsbaum and the following equalities:
\begin{align*}
\chi(J)=\lim_{t\to 1} \sum_{i\ge 0} (-1)^i P(H_i,t)&=\lim_{t\to 1} \sum_{i\ge 0}(-1)^i  P(K_i,t)\\
&=\lim_{t\to 1}\prod_{i=1}^d(1-t^{b_i}) P(R,t)\\
&=b_1 \cdots b_d \lim_{t \to 1} (1-t)^d P(R, t).
\end{align*}
\end{proof}

Note that if $R$ is standard graded over a field, then the value
$\lim_{t\to 1} (1-t)^d P(R,t)$ can be computed from this theorem by
choosing a homogeneous system of parameters of degree 1. This means
all $b_i = 1$, and one sees that $\lim_{t\to 1} (1-t)^d P(R,t) = e(\MM)$, the
multiplicity of the irrelevant ideal.

\medskip

\begin{disc}\label{rplusdef} {\rm We need some of the results of the paper \cite{HH2} which
we discuss here. Let $R$ be an $\mathbb N$-graded Noetherian domain over a field $R_0 = k$
of positive characteristic $p$.  Let $\Omega$ be an algebraic closure of the fraction field
of $R$, and let $R^+$ be the integral closure of $R$ in $\Omega$. 
We refer an element $\theta \in \Omega \setminus \{0\}$ as a homogeneous element of $\Omega$ if $\theta$ is a  root of a nonzero polynomial $F(X)\in R[X]$ such that $X$ can be assigned a degree in $\mathbb Q$ making $F$ homogeneous. 
By \cite[Lemma 4.1]{HH2}, this condition is equivalent to saying that the grading on $R$ (uniquely) extends to a grading on $R[\theta]$ indexed by $\Q$.
In this way (see \cite[Lemma 4.1]{HH2} for the detail), the homogeneous elements in $\Omega$ span a domain graded by $\mathbb Q$, extending
the grading on $R$. 
The homogeneous elements of $R^+$ with degrees in $\mathbb N$ span a subring
of $R^+$ graded by $\mathbb N$. This ring is denoted by $R^{+\mathrm{gr}}$. The inclusion of
$R\subset R^{+\mathrm{gr}}$ is a graded inclusion of degree $0$. For an ideal $I\subseteq R$, we let
$I^{+\mathrm{gr}} = IR^{+\mathrm{gr}}\cap R$.  Part of the main theorem of \cite{HH2}, Theorem
5.15, gives the following statement:

\begin{thm}\label{rplus} 
Let $R$ be a Noetherian $\mathbb N$-graded domain over a field $R_0 = k$ of characteristic $p>0$. 
Every homogeneous system of parameters for $R$ is a regular sequence in $R^{+\mathrm{gr}}$. 
\end{thm}
}
\end{disc}

Again let $R$ be a Noetherian $\N$-graded domain over a field $R_0=k$ of characteristic $p>0$, and let $x_1, \dots, x_d$ be homogeneous elements of $R$. 
Set $A = k[X_1,...,X_d]$.
Consider the natural map $f:A\rightarrow R$ taking $X_i$ to $x_i$. 
Let $I$ be an ideal of $R$ generated by monomials in $x_1,...,x_d$, and let $L\subseteq A$ be the corresponding ideal of monomials in $A$. 
Let $T_{\bullet}$ denote the Taylor resolution of $A/L$.
(see \cite{Ei}). 

The main fact we need is the following:

\begin{lem}\label{flat} In the situation above, if 
$x_1,...,x_d$ are a homogeneous system of  parameters for $R$, and $S = R^{+\mathrm{gr}}$, then $T_{\bullet}\otimes_AS$ is exact. 
In fact, $S$ is flat over $A$.
\end{lem}

\begin{proof} Since $T_{\bullet}$ is a free $A$-resolution of $A/L$, where $L$ is generated by monomials, the first
statement follows from the second statement. The flatness of $S$ over $A$ is proved as in the discussion in
\cite[6.7]{HH}. See also \cite[Theorem 9.1]{Hu} and  \cite[proof of 2.3]{Ho}.
\end{proof}

\begin{disc}\label{colon}
{\rm One of the consequences of $S$ being flat over $A$ is that we can compute the
quotient of monomial ideals in the $x_i$ in $S$ as if the monomials belong to $A$. Specifically,
let $I$ and $I'$ be two ideals in $R$ generated by monomials in $x_1,...,x_d$, which are a
homogeneous system of parameters for $R$. Let $L$ and $L'$ be the corresponding ideals in $A$
generated by the same monomials in $X_1,...,X_d$. Then $IS:_SI' = (L:_AL')S$. This follows
at once from the flatness of $S$ over $A$. Note that we \it cannot \rm replace $S$ by $R$ in
this equality unless $R$ is Cohen-Macaulay.}
\end{disc}

We can now prove:

\begin{thm}\label{mainthmcharp}
Let $R=\bigoplus_{n \ge 0} R_n$ be a $d$-dimensional Noetherian graded
ring with $R_0$ a field of characteristic $p>0$. 
Suppose that $\a$ and $J$ are ideals each generated by a full homogeneous system of parameters for $R$, and  put $a$ to be the least degree of the generators of $\a$. 
Let $c \in R^{\circ}$ be a homogeneous element of degree $n$ such that $c(R/P_i) \not\subseteq (JR/P_i)^{+\mathrm{gr}}$ for all $i$, where the $P_i$ range over the minimal primes of $R$ with $R/P_i$ having the same dimension as $R$.  
If $N \ge 0$ is an integer such that $c\a^{N+1}(R/P_i) \subseteq (JR/P_i)^{+\mathrm{gr}}$ for all $i$, then 
$$e(\a) \ge \left(\frac{ad}{aN+ad+n}\right)^d e(J).$$
In particular, if $\a^{N+1}(R/P_i) \subseteq (JR/P_i)^{+\mathrm{gr}}$ for all $i$, then 
$$e(\a) \ge \left(\frac{d}{N+d}\right)^d e(J).$$
\end{thm}

\begin{proof}
By Remark~\ref{reduce} we may assume without loss of generality that
$R$ is a domain. Throughout the proof, we set $S = R^{+\mathrm{gr}}$.

We chiefly follow the same proof that was given in \cite[Theorem 5.8]{HMTW}, making
suitable modifications to take advantage of the fact that $S$ is a big Cohen-Macaulay
algebra for $R$; however, we have a simplification at the end of the proof.

Suppose that $\a$ is generated by a full homogeneous system of
parameters $x_1, \dots, x_d$ of degrees $a_1 \le \dots \le a_d$, and
that $J$ is generated by another homogeneous system of parameters
$f_1, \dots, f_d$ of degrees $b_1\leq\dots\leq b_d$. 

To prove the theorem it suffices to show that\footnote{In the proof
of \cite[Theorem 5.8]{HMTW} it is claimed that in the Cohen-Macaulay
case the multiplicity of $x_1, \dots, x_d$ (resp. $f_1, \dots, f_d$)
is $a_1 \dots a_d$ (resp. $b_1 \dots b_d$). This is not correct;
however, in that paper we used only  that $\frac {e(x_1,...,x_d)}{a_1 \dots a_d} = \frac {e(f_1,...,f_d)}{b_1 \dots b_d}$.
That this is true is easy to see in the Cohen-Macaulay case, and follows more generally
by Proposition~\ref{mult}.} 
$$(N+d+\frac{n}{a_1})^d a_1 \dots a_d \ge d^d b_1 \dots b_d.$$
That this inequality is enough to prove the theorem follows from Proposition~\ref{mult} since $\frac{e(\a)} {a_1\cdots a_d} =  
\frac{e(J)} {b_1\cdots b_d}$ (this common ratio is 
$\lim_{t\to 1} (1-t)^d P(R,t)$ where $P(R,t)$ is the Poincar\'e series of $R$).

Define positive integers $t_1, \dots, t_{d}$ inductively as follows:
$t_1$ is the smallest integer $t$ such that $cx_1^t \in J^{+\mathrm{gr}} (=  JS\cap R)$. If $2\leq
i\leq d$, then $t_i$ is the smallest integer $t$ such that
$cx_1^{t_1-1} \cdots x_{i-1}^{t_{i-1}-1}x_i^t \in J^{+\mathrm{gr}}$. 

We  show the following inequality for every $i=1, \dots, d$:
\begin{equation}\label{eq_homogeneous2}
t_1 a_1+\cdots+t_{i} a_i +n \geq b_1+\cdots+b_i.
\end{equation}
Let $I_i$ be the ideal of $R$ generated by $x_1^{t_1}, x_1^{t_1-1}x_2^{t_2}, \ldots, x_1^{t_1-1} \cdots x_{i-1}^{t_{i-1}-1}x_i^{t_i}$ for each $1\le i\le d$. 
We let $L_i$ be the corresponding ideal of monomials in $A$,
where $A = k[X_1,...,X_d]$ is as in the discussion above.
Note that the definition of the integers $t_j$ implies that $cI_iS \subseteq JS$. 
We use the Taylor resolution $T_{\bullet}$ for $A/L_i$. 
After tensoring with $S$, this complex is exact by Lemma~\ref{flat}. 
The multiplication map  by $c$ from $S/I_iS$ into $S/JS$ induces a comparison map of degree $n$ between $T_{\bullet}\otimes_AS$ and the Koszul complex of the generators $f_1,...,f_d$ of $J$. 
This Koszul complex is acyclic since the $f_i$ form a regular sequence in $S$ by Theorem~\ref{rplus}. 
Note that the $i^{\rm th}$ step in the Taylor complex $T_{\bullet}$ for the monomials
$X_1^{t_1},X_1^{t_1-1}X_2^{t_2},\ldots,X_1^{t_1-1}\cdots X_{i-1}^{t_{i-1}-1}X_i^{t_i}$ in  $A$ is a free module of rank one, with a generator corresponding to the monomial
$${\rm lcm}(X_1^{t_1},X_1^{t_1-1}X_2^{t_2},\ldots,X_1^{t_1-1}\cdots
X_{i-1}^{t_{i-1}}X_i^{t_i})=X_1^{t_1}\cdots
X_{i-1}^{t_{i-1}}X_i^{t_i}$$ (see \cite[Exercise 17.11]{Ei}).
It follows that the map of degree $n$ between the $i^{\rm th}$ steps in the resolutions of $S/I_iS$ and $S/JS$ is of the form
$$S(-t_1 a_1-\cdots-t_{i} a_i) \to \bigoplus_{1 \le v_1 < \dots < v_i \le d}S(-b_{v_1}-\dots-b_{v_i}).$$
In particular, unless this map is zero, we have
$$t_1a_1+\cdots+t_{i} a_i+n \geq \min_{1 \le v_1 < \dots < v_i \le
d}\{b_{v_1}+\cdots+b_{v_i}\}=b_1+\cdots+b_i.$$ 
We now show that this map cannot be zero. If it is zero, then also the induced map
\begin{equation}\label{eq_homogeneous1}
\Tor_i^S(S/I_iS, S/\b_i) \to \Tor_i^S(S/JS, S/\b_i)
\end{equation}
is zero, where $\b_i$ is the ideal in $S$ generated by $x_1, \dots, x_i$. On the other
hand, using the Koszul resolution  on $x_1, \dots, x_i$ to compute the
above $\Tor$ modules (these elements are a regular  sequence in $S$), we see that the map (\ref{eq_homogeneous1})
can be identified with the multiplication map by $c$ 
$$(I_iS\colon_S \b_i)/I_iS \xrightarrow{\times c} (JS\colon_S \b_i)/JS.$$
Since $x_1^{t_1-1}\cdots x_i^{t_i-1}\in (I_iS\colon \b_i)$, if the map in (2) is zero, 
then it follows that
$cx_1^{t_1-1} \cdots x_i^{t_i-1}$
lies in $J^{+\mathrm{gr}}$, a contradiction. This proves (\ref{eq_homogeneous2}).

To finish the proof, we will use the following claim, which is a slight modification of the claim in the proof of \cite[Theorem 5.8]{HMTW}:
\begin{cl}
Let $\alpha_i,\beta_i,\gamma_i$ be real numbers for $1\leq i\leq d$ and $\omega$ be a real number. 
 If $1=\gamma_1\leq\gamma_2\leq\ldots\leq\gamma_d$ and if
$\gamma_1\alpha_1+\dots+\gamma_i\alpha_i+\omega \ge \gamma_1\beta_1+\dots+\gamma_i\beta_i$ for all $i=1, \dots, d$, then $\alpha_1+\dots+\alpha_d+\omega \ge \beta_1+\dots+\beta_d$.
\end{cl}
\begin{proof}[Proof of Claim]
The proof is essentially the same as that of the claim in the proof of \cite[Theorem 5.8]{HMTW}. 
\end{proof}

As in \cite{HMTW},\footnote{The end of the proof is simpler here than in \cite{HMTW},
due to the fact that we are able to remove an argument using linkage. This improvement
comes from a suggestion by {\bf Hailong Dao}, who we thank.} we now
set $\alpha_i=t_i$, $\beta_i=b_i/a_i$ and $\gamma_i=a_i/a_1$ for $1 \le i \le d$. 
We put $\omega=n/a_1$. 
Since $a_1\leq\dots\leq
a_d$, we deduce $1=\gamma_1\leq\dots\leq\gamma_d$. Moreover, using
(\ref{eq_homogeneous2}), we obtain that
$\gamma_1\alpha_1+\dots+\gamma_i\alpha_i+\omega \ge
\gamma_1\beta_1+\dots+\gamma_i\beta_i$ for $1 \le i \le d$. Using
the above claim, we conclude that
$$t_1+\dots+t_d+\frac{n}{a_1}=\alpha_1+\dots+\alpha_d+\omega \ge \beta_1+\dots +\beta_d=\frac{b_1}{a_1}+\dots+\frac{b_d}{a_d}.$$
The inductive definition of the $t_i$ shows that the monomial
$cx_1^{t_1-1}\cdots x_d^{t_d-1}\notin JS$. 
Since $c\a^{N+1}\subseteq JS$, we
see that $N+d \ge t_1+\ldots + t_d$. Comparing the arithmetic mean
with the geometric mean of the $b_i/a_i$ yields the conclusion that
$$(N+d+\frac{n}{a_1})^d a_1 \dots a_d \ge d^d b_1 \dots b_d,$$
which finishes the proof of Theorem~\ref{mainthmcharp}.
\end{proof}

\begin{rem}\label{equality}
Let the notation be as in Theorem \ref{mainthm}. 
Since we compare the arithmetic mean with the geometric mean of the $\frac{b_i}{a_i}$ in the proof of Theorem \ref{mainthmcharp}, if equality holds in the inequality of Theorem \ref{mainthm}, then $(a_1, \dots, a_d)$ and $(b_1, \dots, b_d)$ have to be proportional, that is, $\frac{b_1}{a_1}= \dots =\frac{b_d}{a_d}$. 
We also remark that if equality holds in the inequality of Corollary \ref{maincor}, then $(a_1, \dots, a_d)$ and $(b_1, \dots, b_d)$ also have to be proportional. 
Suppose to the contrary that $(a_1, \dots, a_d)$ and $(b_1, \dots, b_d)$ are not proportional. In this case, the rational number 
$$\epsilon:=\frac{({\frac{b_1}{a_1}+\dots+\frac{b_d}{a_d}})^da_1 \cdots a_d}{d^d b_1\cdots b_d}-1$$ 
is strictly positive. For all $q=p^e$, one has  
$$\left({\frac{b_1q}{a_1}+\dots+\frac{b_dq}{a_d}}\right)^da_1 \cdots a_d=(1+\epsilon)d^d(b_1q)\cdots (b_dq).$$
By the argument as in the proof of Theorem \ref{mainthmcharp}, this implies that 
$$e(\a) \geq (1+\epsilon) \left(\frac{d}{\nu^J_{\a}(q)+d}\right)^d e(J^{[q]})=(1+\epsilon) \left(\frac{qd}{\nu^J_{\a}(q)+d}\right)^d e(J)$$ for all $q=p^e$.
Since $\epsilon$ is independent of $q$, we conclude that 
$$e(\a) \geq (1+\epsilon) \left(\frac{d}{\pt_-^J(\a)}\right)^d e(J) \gneq \left(\frac{d}{\pt_-^J(\a)}\right)^d e(J) .$$
\end{rem}
\begin{rem} If equality holds in the inequality of Corollary 
\ref{maincor}, we think $\a$ and $J$ are \lq\lq equivalent" in the sense of 
the following conjecture, which is true if $d=2$ or, more generally, 
if the set $\{a_1,\ldots ,a_d\}$ consists of at most $2$ elements.  
\end{rem}

\begin{conj} 
Let $R$, $\a$ and $J$ be as in Corollary \ref{maincor}.
Suppose that 
$$e(\a) = \left(\frac{d}{\pt^J_{-}(\a)}\right)^d e(J).$$
If we put  $b_1/a_1=\ldots = b_d/a_d =  t/s$, where $s$ and $t$ are positive integers, then $\a^t$ and $J^s$ have the same integral closure. 
\end{conj}

\bigskip

\section{Main Theorem in Characteristic Zero}

\medskip

In this section we outline the proof of Theorem~\ref{mainthm} in characteristic zero. 
The reduction to characteristic $p$ is fairly standard.
We prove:

\begin{thm}\label{mainthm0}
Let $R=\bigoplus_{n \ge 0} R_n$ be a $d$-dimensional Noetherian graded
ring with $R_0 = k$ a field of characteristic zero.
Suppose that $\a$ and $J$ are ideals each generated by a full homogeneous system of
parameters for $R$.  
If $\a^{N+1}\subseteq J$ for some integer $N \ge 0$, then
$$e(\a) \ge \left(\frac{d}{d+N}\right)^d e(J).$$
\end{thm}

\begin{proof} Let $\a$ be generated by the homogeneous parameters $x_1,...,x_d$ and
let $J$ be generated by the homogeneous parameters $f_1,...,f_d$. Without loss of
generality we may assume that the degree of $x_i$ is $a_i$ with $a_1\le a_2\ldots \le a_d$, 
and similarly that the degree of $f_i$ is $b_i$ with $b_1\le b_2\le \ldots \le b_d$.
As in the proof of Theorem \ref{mainthmcharp}, 
it suffices to prove that 
$$a_1\cdots a_d  \ge \left(\frac{d}{d+N}\right)^db_1\cdots b_d.$$
We begin by writing $R = k[t_1,...,t_n]\cong  B/I$, where $B = k[T_1,...,T_n]$ is a
homogeneous polynomial ring with each $T_i$ having a positive degree, such that
$I$ is homogeneous and the isomorphism of $R$ with $B/I$ is degree-preserving, taking each $T_i$
to the homogenous algebra generator $t_i$ of $R$. We lift each $x_i$ to a homogeneous polynomial $h_i\in B$
and each $f_i$ to a homogeneous polynomial $g_i\in B$. We also fix homogeneous generators
$F_1,...,F_l$ of $I$. Furthermore, since the maximal homogeneous ideal of $R$ has a power
contained in the ideal $(x_1,...,x_d)$, there are equations in $B$ which express
this. A typical one, e.g., would be of the form
$$T_j^M = \sum_{1\leq i\leq d} p_{ji}h_i + \sum_{1\leq i\leq l} q_{ji}F_i$$
for some fixed large power $M$.

Since the maximal homogeneous ideal of $R$ also has a power contained in the ideal $(f_1, \dots, f_d)$, 
there is another set of equations expressing the fact that the $T_j$ are nilpotent modulo
the ideal $I$ plus the ideal generated by $g_1,...,g_d$. 
There are also equations which express the fact that each monomial of total degree
$N$ in the $x_i$ is in $J$; these are expressed by equations which give that
every monomial $h_1^{m_1}\cdots h_d^{m_d}$ in the $h_i$ of total degree $N = m_1+\ldots + m_d$ is equal to an element in the
ideal $I + (g_1,..,g_d)B$.

Now let $A$ be a finitely generated $\mathbb Z$-subalgebra of $k$ which has the
coefficients of all polynomials in all the equations and defining ideals above. 
We let $B_A = A[T_1,...,T_n]$. By possibly adding more coefficients and generators, we can assume that
$I_A = I\cap B_A$ is generated by $F_1,...,F_l\in B_A$.  We let $x_{iA}$ be the image of
$h_i$ in $R_A = B_A/I_A$, and let $f_{iA}$ be the image of $g_i$ in the same ring. By the lemma of generic flatness, we can invert one element of $A$ to make
$R_A$ free over $A$ (here we replace $A$ by the localization of $A$ at that one element. It is
still finitely generated over $\mathbb Z$). Note that $R = R_A\otimes_Ak$.

Choose any
maximal ideal $\n$ of $A$, and let $\kappa = A/\n$, a finite field. We use $\kappa$ as a subscript to
denote images after tensoring over $A$ with $\kappa$. The dimension $d$ of $R$ is
equal to the dimension of $R_{\kappa}$, and $R_{\kappa}$ is a positively graded 
Noetherian ring over
the field $\kappa$. For the proofs and a discussion of the dimension 
of the fiber $R_{\kappa}$, see \cite[Appendix, Section 2]{Hu} or \cite[2.3]{HH3}.

The images of the $x_{iA}$ and the $f_{iA}$ in $R_{\kappa}$ form homogeneous systems of
parameters generating ideals $\a_{\kappa}$ and $J_{\kappa}$ respectively, and
furthermore $\a_{\kappa}^N\subseteq J_{\kappa}$.  Moreover the degrees of these elements are
the same as the degrees of the corresponding elements in characteristic zero. By the main theorem in characteristic $p$ (Theorem \ref{mainthmcharp}), 
$$a_1\cdots a_d\ge  \left(\frac{d}{d+N}\right)^d b_1\cdots b_d.$$
\end{proof}

\bigskip

\section{A comparison of $F$-thresholds and $F$-jumping numbers}

\medskip

In this section, we compare $F$-thresholds and jumping numbers for the generalized parameter test submodules introduced in \cite{ST}. 
Throughout this section, we use the following notation. 

\begin{notation}
Let $(R, \m)$ be a $d$-dimensional Noetherian excellent reduced ring of equal characteristic satisfying one of the following conditions. 
\begin{enumerate}[{(}a{)}]
\item $R$ is a complete local ring with the maximal ideal $\m$. 
%which is a homomorphic image of a Gorenstein local ring. 
\item $R=\bigoplus_{n \ge 0}R_n$ is a graded ring with $R_0$ a field and $\m=\bigoplus_{n \ge 1}R_n$. 
\end{enumerate} 
Then $R$ admits a (graded) \textit{canonical module} $\omega_R$: 
in the case (a), 
$\omega_R$ is the finitely generated $R$-module $\Hom_R(H^d_{\m}(R), E_R(R/\m))$, 
%$\omega_R$ is a finitely generated $R$-module such that 
%$\Hom_R(\omega_R, E_R(R/\m)) \cong H^d_{\m}(R)$, 
where $E_R(R/\m)$ is the injective hull of the residue field $R/\m$. 
In the case (b), 
$\omega_R$ is the finitely generated graded $R$-module $\underline{\mathrm{Hom}}_{R_0}(H^d_{\m}(R), R_0):=\bigoplus_{n \in \Z}\Hom_{R_0}([H^d_{\m}(R)]_{-n}, R_0)$. 
%$\omega_R$ is a finitely generated graded $R$-module such that 
%$$\underline{\mathrm{Hom}}_{R_0}(\omega_R, R_0):=\bigoplus_{n \in \Z}\Hom_{R_0}([\omega_R]_{-n}, R_0) \cong H^d_{\m}(R).$$

Also, in dealing with graded rings, we assume that all the ideals and systems of parameters considered are homogeneous. 
\end{notation}

First we recall the definition of a generalization of tight closure introduced by Hara and Yoshida \cite{HY}. 
\begin{defn}
Assume that $R$ is a ring of characteristic $p>0$. 
Let $\a$ be an ideal of $R$ such that $\a \cap R^{\circ} \ne \emptyset$ and $t \ge 0$ be a real number. 

\begin{enumerate}[{(}i{)}]
\item (\cite[Definition 6.1]{HY})
For any ideal $I \subseteq R$, the $\a^t$-\textit{tight closure} $I^*$ of $I$ is defined to be the ideal of $R$ consisting of all elements $x \in R$ for which there exists $c \in R^{\circ}$ such that 
$$c\a^{\lceil tq \rceil}x^q \subseteq  I^{[q]}$$
for all large $q=p^e$. 
When $\a=R$, we denote this ideal simply by $I^*$. 

\item (\cite[Definition 6.1]{HY})
The $\a^t$-\textit{tight closure} $0^{*\a^t}_{H^d_{\m}(R)}$ of the zero submodule in $H^d_{\m}(R)$ is defined to be the submodule of $H^d_{\m}(R)$ consisting of all elements $\xi \in H^d_{\m}(R)$ for which there exists $c \in R^{\circ}$ such that 
$$c\a^{\lceil tq \rceil}\xi^q=0$$
in $H^d_{\m}(R)$ for all large $q=p^e$, where $\xi^q:=F^e(\xi) \in H^d_{\m}(R)$ denote the image of $\xi$ via the induced $e$-times iterated Frobenius map $F^e: H^d_{\m}(R) \to H^d_{\m}(R)$. 
When $\a=R$, we denote this submodule simply by $0^*_{H^d_{\m}(R)}$. 

\item (\cite[Definition 6.6]{ST})
An element $c \in R^{\circ}$ is called a \textit{parameter $\a^t$-test element} if, for all ideals $I$ generated by a system of parameters for $R$, one has $c\a^{\lceil tq \rceil}x^q \subseteq I^{[q]}$ for all $q=p^e$ whenever $x \in I^{*\a^t}$. 
When $\a=R$, we call such an element simply as a parameter test element. 

\item (\cite{FW}) 
$R$ is said to be \textit{$F$-rational} if $I^*=I$ for all ideals $I \subseteq R$ generated by a system of parameters for $R$. 
This is equivalent to saying that $R$ is Cohen-Macaulay and $0^*_{H^d_{\m}(R)}=0$. 
\end{enumerate}
\end{defn}

Now we are ready to state the definition of $F$-jumping numbers introduced in \cite{ST}.
 
\begin{defn}\label{parameter $F$-jumping}
Assume that $R$ is a ring of characteristic $p>0$, and let $\a$ be an ideal of $R$ such that $\a \cap R^{\circ} \ne \emptyset$. 

\begin{enumerate}[{(}i{)}]

\item (\cite[Remark 6.4]{ST})
For every real number $t \ge 0$, the \textit{generalized parameter test submodule} $\tau(\omega_R, \a^t)$ associated to the pair $(R,\a^t)$ is defined to be
$$\tau(\omega_R, \a^t)=\Ann_{\omega_R}(0^{*\a^t}_{H^d_{\m}(R)}) \subseteq \omega_R.$$
\item (\cite[Definition 7.9]{ST}) 
For every ideal $J \subseteq R$ such that $\a \subseteq \sqrt{J}$, the \textit{$F$-jumping number} $\mathrm{fjn}^J(\omega_R, \a)$ of $\a \omega_R$ with respect to $J \omega_R$ is defined to be
$$\mathrm{fjn}^J(\omega_R, \a)=\inf\{t \ge 0 \mid \tau(\omega_R, \a^t) \subseteq J \omega_R\}.$$
\end{enumerate}
\end{defn}

\begin{rem}
 (1) (\cite[Lemma 7.10]{ST}) For every real number $t \ge 0$, there exists $\varepsilon>0$ such that 
$\tau(\omega_R, \a^t)=\tau(\omega_R, \a^{t+\varepsilon})$. In particular, 
$$\mathrm{fjn}^J(\omega_R, \a)=\min\{t \ge 0 \mid \tau(\omega_R, \a^t) \subseteq J \omega_R\}.$$
 
 (2) (\cite[Proposition 2.7]{MTW}) In case $R$ is a regular ring, the $F$-jumping number $\mathrm{fjn}^J(\omega_R, \a)$ coincides with the $F$-threshold $\pt^J(\a)$ for all ideals $\a$ and $J$ such that $\a \subseteq \sqrt{J}$.

(3) (\cite[Remark 2.5]{HMTW}) In general, the $F$-jumping number $\mathrm{fjn}^J(\omega_R, \a)$ disagrees with the $F$-threshold $\pt^J(\a)$. 
\end{rem}

If $J$ is generated by a full system of parameters and if the ring is $F$-rational away from $\a$, then the $F$-jumping number $\mathrm{fjn}^J(\omega_R, \a)$ coincides with the $F$-threshold $\pt^J(\a)$. 
\begin{thm}\label{characterization}
Suppose that $R$ is an equidimensional ring of characteristic $p>0$ and $J$ is an ideal generated by a full system of parameters for $R$. 
Assume in addition that 
%$V(\a)$ contains the non-$F$-rational locus of $\Spec R$ $($which is equivalent to saying that
$R_P$ is $F$-rational for all prime ideals $P$ not containing $\a$.
Then 
$$\mathrm{fjn}^J(\omega_R, \a)=\pt_{+}^J(\a).$$
\end{thm}

\begin{proof}
Suppose that $J$ is generated by a full system of parameters $x_1, \dots, x_d$ for $R$. 
Let $\xi=\left[\frac{1}{x_1 \cdots x_d}\right] \in H^d_{\m}(R)$. It is clear that $\xi \in (0 :_{H^d_{\m}(R)} J \omega_R)$. 

Suppose that $\tau(\omega_R, \a^t) \subseteq J \omega_R$.
Since 
$\xi \in (0:_{H^d_{\m}(R)} J\omega_R) \subseteq 0^{*\a^t}_{H^d_{\m}(R)}$, 
there exists $c \in R^{\circ}$ such that  
$c \a^{\lceil tq \rceil}\xi^q=0$ in $H^d_{\m}(R)$ 
for all large $q=p^e$.  
This implies that there exists $s \in \N$ such that 
$c\a^{\lceil tq \rceil}(x_1 \cdots x_d)^s \subseteq (x_1^{q+s}, \dots, x_d^{q+s})$. 
We then use the colon-capturing property of tight closure \cite[Theorem 7.15 (a)]{HH} to have $c\a^{\lceil tq \rceil} \subseteq (x_1^q, \dots, x_d^q)^*$. 
For any parameter test element $c' \in R^{\circ}$, one has 
$cc' \a^{\lceil tq \rceil} \subseteq (x_1^q, \dots, x_d^q)=J^{[q]}$ for all large $q=p^e$, so that $1 \in J^{*\a^t}$. 
By the definition of parameter $\a^t$-test elements, $c''\a^{\lceil tq \rceil} \subseteq J^{[q]}$ for all parameter $\a^t$-test elements $c'' \in R^{\circ}$ and for all $q=p^e$. 
Here note that it follows from \cite[Lemma 6.8]{ST} there exists an integer $n \ge 1$ such that every element of $\a^n \cap R^{\circ}$ is a parameter $\a^t$-test element. 
Therefore $\a^{\lceil tq \rceil+n} \subseteq J^{[q]}$, that is, $\nu^J_{\a}(q) \le \lceil tq \rceil +n-1$ for all $q=p^e$. 
By dividing by $q$ and taking the limit as $q$ goes to the infinity, we have $t \geq \pt_{+}^J(\a)$. 
Thus, $\mathrm{fjn}^J(\omega_R, \a)\geq \pt_{+}^J(\a)$. 

To prove the converse inequality, suppose that $t >\pt_{+}^J(\a)$. 
Fix any $\eta \in (0 :_{H^d_{\m}(R)} J\omega_R)$. 
Here note that $(0 :_{H^d_{\m}(R)} J\omega_R)=(0:_{H^d_{\m}(R)} J)$, because $\omega_R \times H^d_{\m}(R) \to H^d_{\m}(\omega_R)$ is the duality pairing. 
Hence $J^{[q]}\eta^q=0$ for all $q=p^e$. 
Since $\a^{\lceil tq \rceil} \subseteq J^{[q]}$ for all large $q=p^e$ by the definition of $\pt_{+}^J(\a)$, 
one has $\a^{\lceil tq \rceil}\eta^q=0$ for all large $q=p^e$, that is, $\eta \in 0^{*\a^t}_{H^d_{\m}(R)}$.
Thus, $(0 :_{H^d_{\m}(R)} J\omega_R) \subseteq 0^{*\a^t}_{H^d_{\m}(R)}$, or equivalently, $J\omega_R \supseteq \tau(\omega_R, \a^t)$. 
Taking infimum of such $t$'s, we conclude that $\mathrm{fjn}^J(\omega_R, \a) \leq \pt_{+}^J(\a)$.  
\end{proof}

\begin{ques}
Does the equality in Theorem \ref{characterization} still hold true if $V(\a)$ does not contain the non-$F$-rational locus of $\Spec R$? 
We have seen in the proof of Theorem \ref{characterization} that the inequality $\mathrm{fjn}^J(\omega_R, \a)\leq \pt_{+}^J(\a)$ holds even in this case. 
\end{ques}

We can replace the $F$-threshold $\pt^J(\a)$ with the $F$-jumping number $\mathrm{fjn}^J(\omega_R, \a)$ in the inequality of Corollary \ref{maincor}. 
The following is a corollary of Theorem \ref{mainthmcharp}. 

\begin{cor}\label{maincor2}
Let $R=\bigoplus_{n \ge 0} R_n$ be a $d$-dimensional Noetherian equidimensional reduced graded ring with $R_0$ a field of characteristic $p>0$.
Suppose that $\a$ $($resp. $J)$ is an ideal generated by a full homogeneous system of
parameters of degrees $a_1 \leq \dots \leq a_d$ $($resp. $b_1 \leq \dots \leq b_d)$ for $R$.  
Then 
$$e(\a) \ge \left(\frac{d}{\mathrm{fjn}^J(\omega_R, \a)}\right)^d e(J).$$
If the equality holds in the above, then $(a_1, \dots, a_d)$ and $(b_1, \dots, b_d)$ are proportional, that is, $a_1/b_1=\dots=a_d/b_d$. 
\end{cor}
\begin{proof}
Suppose that $\tau(\omega_R, \a^t) \subseteq J \omega_R$. 
By the same argument as the proof of Theorem \ref{characterization}, there exists a homogeneous element $c \in R^{\circ}$ of degree $n$ such that $c\a^{\lceil tq \rceil} \subseteq J^{[q]}$ for all $q=p^e$. 
Here note that each $J^{[q]}$ is again generated by a full homogeneous system of parameters for $R$.  
It follows from Theorem \ref{mainthmcharp} that for all large $q=p^e$, we have
$$e(\a)\geq\left(\frac{a_1d}{a_1\lceil tq \rceil+a_1d+n-a_1}\right)^de(J^{[q]})=\left(\frac{a_1dq}{a_1\lceil tq \rceil+a_1d+n-a_1}\right)^d e(J).$$ 
The right-hand side converges to $\left(\frac{d}{t}\right)^d e(J)$ as $q$ goes to the infinity.
Taking infimum of such $t$'s, we obtain the inequality in the corollary. 

The latter assertion follows from a similar argument to Remark \ref{equality}. 
\end{proof}

In general, we cannot replace $\mathrm{fjn}^J(\omega_R, \a)$ by 
$$\mathrm{fjn}^J(R, \a)=\mathrm{fjn}^J(\a):=\inf\{t \ge 0 \mid \tau(R, \a^t) \subseteq J\}$$
in the inequality of Corollary \ref{maincor2} (see \cite{HY} for the definition of generalized test ideals $\tau(\a^t)$). 

\begin{eg}\label{3rd veronese}
Let $S=\F_p[x,y]$ be the two-dimensional polynomial ring over the finite field $\F_p$ and $R=S^{(3)}$ be the third Veronese subring of $S$. 
Let $J=(x^3, y^3)$ be a parameter ideal of $R$ and $\m_S=(x,y)$ (resp. $\m_R=(x^3, x^2y, xy^2, y^3)$) be a maximal ideal of $S$ (resp. $R$). 
Note that 
$$\tau(R, J^t)=\tau(R, \m_R^t)=\tau(S, \m_S^{3t}) \cap R=\m_S^{\lfloor 3t \rfloor-1} \cap R$$ 
for $t \geq 1/3$, where the second equality follows from \cite[Theorem 3.3]{HT}. 
Since $\m_R^2 \subseteq J$, we have $\mathrm{fjn}^J(J)=5/3 < 2=\pt^J(J)=\mathrm{fjn}^J(\omega_R, \a)$.
\end{eg}

We can define a geometric analogue of $F$-jumping numbers. 
\begin{defn}\label{jump def}
Assume that $R$ is a normal domain of essentially of finite type over a field of characteristic zero, and let $\a$ be an ideal of $R$ such that $\a \cap R^{\circ} \ne \emptyset$. 
\begin{enumerate}[{(}i{)}]
\item (\cite[Definition 2]{Bl})
Let $\pi: \widetilde{X} \to X:=\Spec R$ be a log resolution of $\a$ such that $\a \sO_{\widetilde{X}}=\sO_{\widetilde{X}}(-F)$. 
For every real number $t \ge 0$, the \textit{multiplier submodule} $\J(\omega_R, \a^t)$ associated to the pair $(R,\a^t)$ is defined to be
$$\J(\omega_R, \a^t)=\pi_*(\omega_{\widetilde{X}} \otimes \sO_{\widetilde{X}}(\lceil -tF \rceil)) \subseteq \omega_R.$$

\item For every ideal $J \subseteq R$ such that $\a \subseteq \sqrt{J}$, the \textit{jumping number} $\lambda^J(\omega_R, \a)$ of $\a \omega_R$ with respect to $J \omega_R$ is defined to be
$$\lambda^J(\omega_R, \a)=\inf\{t \ge 0 \mid \J(\omega_R, \a^t) \subseteq J \omega_R\}.$$
\end{enumerate}
\end{defn}

By virtue of \cite[Remark 6.12]{ST}, we obtain the following correspondence between jumping numbers and $F$-jumping numbers. 
\begin{prop}\label{correspondence}
Let $R, \a, J$ be as in Definition \ref{jump def} and $(R_p, \a_p, J_p)$ be a reduction of $(R,\a, J_p)$ to sufficiently large characteristic $p \gg 0$. 
Then we have the followings. 

\begin{enumerate}[${(}1{)}$]
\item $\mathrm{fjn}^{J_p}(\omega_{R_p}, \a_p) \le \lambda^J(\omega_R, \a)$. 

\item $\lim_{p \to \infty} \mathrm{fjn}^{J_p}(\omega_{R_p}, \a_p)=\lambda^J(\omega_R, \a)$. 

\end{enumerate}
\end{prop}

The following is immediate from Corollary \ref{maincor2} and Proposition \ref{correspondence} by  an argument similar to the proof of Theorem \ref{mainthm0}. 
This can be viewed as a generalization of \cite[Theorem 0.1]{dFEM} in the graded case. 
\begin{cor}\label{maincor3}
Let $R=\bigoplus_{n \ge 0} R_n$ be a $d$-dimensional Noetherian normal graded domain with $R_0$ a field of characteristic zero.
Suppose that $\a$ $($resp. $J)$ is an ideal generated by a full homogeneous system of
parameters of degrees $a_1 \leq \dots \leq a_d$ $($resp. $b_1 \leq \dots \leq b_d)$ for $R$.  
Then 
$$e(\a) \ge \left(\frac{d}{\lambda^J(\omega_R, \a)}\right)^d e(J).$$
If the equality holds in the above, then $(a_1, \dots, a_d)$ and $(b_1, \dots, b_d)$ are proportional, that is, $a_1/b_1=\dots=a_d/b_d$. 
\end{cor}

\begin{rem}
Let $R=\bigoplus_{n \ge 0} R_n$ be a $d$-dimensional Noetherian normal graded domain with $R_0$ a field of positive characteristic (resp. a field of characteristic zero). 
Suppose that $\a$ and $J$ are ideals each generated by a full homogeneous system of parameters for $R$. 
By a Skoda-type theorem for $\tau(\omega_R, \a^t)$ (resp. $\J(\omega_R, \a^t)$) (see \cite[Lemma 7.10 (3)]{ST}), we have
$$\tau(\omega_R, \a^t)=\a\tau(\omega_R, \a^{t-1}) \quad (\textup{resp. }\J(\omega_R, \a^t)=\a \J(\omega_R, \a^{t-1}))$$
for all real numbers $t \ge d$. 
Hence, if $\a^{N+1} \subseteq J$ for some integer $N \ge 0$, then
\begin{align*}
\tau(\omega_R, \a^{d+N})&=\a^{N+1}\tau(\omega_R, \a^{d-1}) \subseteq J \omega_R,\\
(\textup{resp. }\J(\omega_R, \a^{d+N})&=\a^{N+1}\J(\omega_R, \a^{d-1}) \subseteq J \omega_R)
\end{align*}
so that $\mathrm{fjn}^J(\omega_R, \a) \leq d+N$ (resp. $\lambda^J(\omega_R, \a) \leq d+N$). 
Thus, we can think of Corollaries \ref{maincor2} and \ref{maincor3} as a generalization of Theorem \ref{mainthm} when the ring is a normal domain. 
\end{rem}

\begin{small}
\begin{acknowledgement}
We thank Mircea  Musta\c{t}\u{a} and Ken-ichi Yoshida  for many helpful discussions concerning the material in this paper. 
We are also grateful to Hailong Dao and Kazuhiko Kurano for valuable suggestions. 
The first author was partially supported by NSF grant DMS-0756853.
The second and third and authors were partially supported by Grant-in-Aid for Scientific Research 20740019 and 20540050, respectively. 
The second author was also partially supported by Program for Improvement of Research Environment for Young Researchers from SCF commissioned by MEXT of Japan.
\end{acknowledgement}
\end{small}

\end{document}